\documentclass{amsart}

\usepackage[english]{babel}
\usepackage{mathrsfs}
\usepackage{amsmath}
\usepackage{amsfonts}
\usepackage{amssymb}
\usepackage{amsthm}
\usepackage{amsrefs}
\usepackage{enumerate}
\usepackage{url}

\newtheorem{thm}{Theorem}[section]
\newtheorem{rmk}[thm]{Remark}
\newtheorem{prop}[thm]{Proposition}
\newtheorem{cor}{Corollary}[thm]
\newtheorem{lema}[thm]{Lemma}
\newtheorem{defi}[thm]{Definition}

\emergencystretch=\maxdimen
\title[Stability of pencils of curves]{Stability of pencils of plane curves, log canonical thresholds and multiplicities}
\author{Aline Zanardini}
\address{University of Pennsylvania, Department of Mathematics, USA}
\email{alinez@math.upenn.edu}
\date{\today}

\begin{document}

\begin{abstract}
In this paper we study the problem of classifying pencils of curves of degree $d$ in $\mathbb{P}^2$ using geometric invariant theory. We consider the action of $SL(3)$ and we relate the stability of a pencil to the stability of its generators, to the log canonical threshold of its members, and to the multiplicities of its base points, thus obtaining explicit stability criteria.
\end{abstract}

\maketitle

\section{Introduction}

This work fits into a series of papers \cite{quadricsP4}, \cite{quartics}, \cite{cubicsP3}, \cite{stab}  on the stability (in the sense of geometric invariant theory) of pencils of hypersurfaces of a fixed degree in some projective space up to projective equivalence; however, the approach we consider here is new.  We focus on curves of degree $d$ in $\mathbb{P}^2$, and we obtain explicit stability criteria in terms of some known invariants of singularities. 

We relate the stability of a pencil $\mathcal{P}$ under the action of $SL(3)$ to the log canonical threshold of pairs $(\mathbb{P}^2,\mathcal{C}_d)$, where $\mathcal{C}_d$ is a curve in $\mathcal{P}$. Moreover, adapting the ideas in \cite{lct}*{Lemma 3.3}, we are also able to relate the stability of a pencil $\mathcal{P}$ to the multiplicities of its base points. In a forthcoming paper \cite{azstab} we will use these criteria and the results obtained in \cite{azconstr}  to provide a complete and geometric characterization of the stability of Halphen pencils of index two under the action of $SL(3)$, in terms of the type of singular fibers appearing in the associated rational elliptic surfaces. 

Letting $\mathscr{P}_d$ denote the space of all pencils of plane curves of degree $d$, our main results are given by Theorems \ref{m1}, \ref{m2} and \ref{m3} below.

\begin{thm}[= Theorem \ref{alpha}]
Let $\mathcal{P}$ be a pencil in $\mathscr{P}_d$ containing a curve $C_f$ such that $lct(\mathbb{P}^2,C_f)=\alpha$. If $\mathcal{P}$ is unstable (resp. not stable), then $\mathcal{P}$ contains a curve $C_g$ such that $lct(\mathbb{P}^2,C_g)< \frac{3\alpha}{2d\alpha-3}$ (resp. $\leq$).
\label{m1}
\end{thm}

\begin{thm}[= Theorem \ref{onequarter}]
If $\mathcal{P}\in \mathscr{P}_d$ is semistable (resp. stable), then $lct_p(\mathbb{P}^2,C_f)\geq \frac{3}{2d}$ (resp. $>$) for any curve $C_f$ in $\mathcal{P}$ and any base point $p$.
\label{m2}
\end{thm}

\begin{thm}[= Theorem \ref{unsmult}]
Let $\mathcal{P}$ be a pencil in $\mathscr{P}_d$. If we can find two generators $C_f$ and $C_g$ of $\mathcal{P}$ such that  $mult_p(C_f) + mult_p(C_g)>\frac{4d}{3}$ (resp. $\geq$) for some base point $p$, then $\mathcal{P}$ is unstable (resp. not stable).
\label{m3}
\end{thm}

In particular, we extend and idea of Hacking \cite{hacking} and Kim-Lee \cite{kimlee} who observed the following  connection between two notions of stability, one coming from geometric invariant theory and the other coming from the Minimal Model Program: if $H\subset \mathbb{P}^n$ is a hypersurface of degree $d$ and the pair $\left(\mathbb{P}^n,\frac{n+1}{d} H\right)$ is log canonical, then $H$ is semistable for the natural action of $PGL(n+1)$. And if $\left(\mathbb{P}^n,(\frac{n+1}{d}+\varepsilon) H\right)$ is log canonical for some $0<\varepsilon \ll 1$, then $H$ is stable. 

One of the ingredients in our approach consists in observing that we can sometimes determine whether a pencil $\mathcal{P}\in \mathscr{P}_d$ is unstable (resp. not stable) or not by looking at the stability of its generators. We also prove Theorems \ref{m4}, \ref{m5} and \ref{m6} below:

\begin{thm}[= Corollary \ref{allmembers}]
If a pencil $\mathcal{P}\in \mathscr{P}_d$ has only semistable (resp. stable) members, then $\mathcal{P}$ is semistable (resp. stable).
\label{m4}
\end{thm}

\begin{thm}[= Theorem \ref{1sss}]
If $\mathcal{P}\in \mathscr{P}_d$ contains at worst one strictly semistable curve (and all other curves in $\mathcal{P}$ are stable), then $\mathcal{P}$ is stable.
\label{m5}
\end{thm}

\begin{thm}[= Theorem \ref{2sss}]
If $\mathcal{P}\in \mathscr{P}_d$ contains at worst two semistable curves $C_f$ and $C_g$ (and all other curves in $\mathcal{P}$ are stable), then $\mathcal{P}$ is strictly semistable if and only if there exists a one-parameter subgroup $\lambda$ (and coordinates in $\mathbb{P}^2$) such that $C_f$ and $C_g$ are both non-stable with respect to this $\lambda$.
\label{m6}
\end{thm}

\subsection*{Organization} The paper is organized as follows: We begin presenting some background material (Section \ref{back}). Then, in Section \ref{sc}, we relate the stability of a pencil to the stability of its generators. In Section \ref{stablct}, we use the notations and results from Section \ref{sc} to relate the stability of a pencil to the log canonical threshold of its members. Finally, in Section \ref{sm} we relate the stability of a pencil to the multiplicities of its base points.

\subsection*{Aknowledgments}
I am grateful to my advisor, Antonella Grassi, for her constant guidance, the many conversations and the numerous suggestions on earlier versions of this paper. This work is part of my PhD thesis and it was partially supported by a Dissertation Completion Fellowship at the University of Pennsylvania. 

\section{Background}
\label{back}

For the convenience of the reader we begin  by presenting the background material that will be needed later.

\subsection{The Log Canonical Threshold}
\label{lct}

We first recall the basic notions concerning log canonical pairs. We refer to \cite{singpairs} for a more detailed exposition. 

Let $X$ be a normal algebraic variety and let $\Delta=\sum d_iD_i \subset X$ be a $\mathbb{Q}$-divisor, i.e. a $\mathbb{Q}$-linear combination of prime divisors.

\begin{defi}
Given any birational morphism $\mu: \tilde{X}\to X$, with $\tilde{X}$ normal, we can write $K_{\tilde{X}} \equiv \mu^*(K_X+\Delta)+\sum a_E E$, where $E\subset \tilde{X}$ are distinct prime divisors, $a_E \doteq a(E,X,\Delta)$ are the discrepancies of $E$ with respect to $(X,\Delta)$ and a non-exceptional divisor $E$ appears in the sum if and only if $E=\mu_{*}^{-1}D_i$ for some $i$ (in that case with coefficient $a(E,X,\Delta)=-d_i$).
\end{defi}

\begin{defi}
A \textbf{log resolution} of the pair $(X,\Delta)$ consists of a proper birational morphism $\mu: \tilde{X}\to X$ such that $\tilde{X}$ is smooth and $\mu^{-1}(\Delta)\cup Exc(\mu)$ is a divisor with global normal crossings.
\label{logres}
\end{defi}

\begin{defi}
We say $(X,\Delta)$ is  \textbf{log canonical} (lc) if $K_X+\Delta$ is $\mathbb{Q}$-Cartier and given any log resolution $\mu:\tilde{X}\to X$ we have $K_{\tilde{X}} \equiv \mu^*(K_X+\Delta)+\sum a_E E$ with all $a_E\geq - 1$. In particular, if $X$ is smooth and $\Delta=d_iD_i$ is simple normal crossings, then $(X,\Delta)$ is log canonical if and only if $d_i\leq 1$ for all $i$.
\end{defi}

\begin{defi}
The number $lct(X,\Delta) \doteq \sup \{\,\,t\,\,;\,\, (X,t\Delta)\,\, \mbox{is log canonical}\}$ is called the \textbf{log canonical threshold} of $(X,\Delta)$. 
\end{defi}

\begin{rmk}
We can also consider a local version, $lct_p(X,\Delta)$, taking the supremum over all $t$ such that $(X,t\Delta)$ is log canonical in an open neighborhood of $p$, where $p\in X$ is a closed point.
\end{rmk}

\subsection{Geometric Invariant Theory}
\label{git}

 We now recall the relevant definitions and results from Geometric Invariant Theory. We point the reader to  \cite{dolgGIT} for more details.

The setup consists of a reductive group $G$ acting on an algebraic variety $X$ and we start by first assuming $X$ is affine. 

\begin{defi}
A point $x \in X$ is said to be semistable for the $G-$action if and only if $0\notin \overline{G\cdot x}$.
\end{defi}

\begin{defi}
A point $x \in X$ is said to be stable for the $G-$action if and only if the following two conditions hold:
\begin{enumerate}[(i)]
\item The orbit $G\cdot x \subset X$ is closed and
\item The stabilizer $G_x \leq G$ is finite
\end{enumerate}
\end{defi}

If $X \hookrightarrow \mathbb{P}^n$ is projective, a point $x\in X$ will be called semistable (resp. stable) if any point $\tilde{x}\in \mathbb{C}^{n+1}$ lying over $x$ is semistable (resp. stable). From now on we assume that this is the case.

\begin{defi}
A one-parameter subgroup of $G$ consists of a non-trivial group homomorphism $\lambda: \mathbb{C}^{\times} \to G$.
\end{defi}

Given a one-parameter subgroup $\lambda: \mathbb{C}^{\times}\to G$ we may regard $\mathbb{C}^{n+1}$ as a representation of $\mathbb{C}^{\times}$. Since any representation of $\mathbb{C}^{\times}$ is completely reducible and every irreducible representation is one dimensional, we can choose a basis $e_0,\ldots,e_n$ of $\mathbb{C}^{n+1}$ so that $\lambda(t)\cdot e_i = t^{r_i}e_i$, for some $r_i \in \mathbb{Z}$. Then, given $x\in X\hookrightarrow \mathbb{P}^n$ we can pick $\tilde{x}\in Cone(X)\subset \mathbb{C}^{n+1}$ lying above $x$ and write $\tilde{x}=\sum x_ie_i$ with respect to this basis so that $\lambda(t)\cdot x \doteq \lambda(t)\cdot \tilde{x}=\sum t^{r_i}x_ie_i$. The weights of $x$ are the set of integers $r_i$ for which $x_i$ is not zero.

\begin{defi}
Given $x \in X$ we define the Hilbert-Mumford weight of $x$ at $\lambda$ to be $\mu(x,\lambda)\doteq \min\{r_i\,\,:\,\,x_i \neq 0\}$.
\label{hmw}
\end{defi}

\begin{rmk}
The Hilbert-Mumford weight satisfies the following properties:
\begin{enumerate}[(i)]
\item $\mu(x,\lambda^{n})=n\mu(x,\lambda)$ for all $n \in \mathbb{N}$
\item $\mu(g\cdot x, g\lambda g^{-1})=\mu(x,\lambda)$ for all $g\in G$
\end{enumerate}
\label{phm}
\end{rmk}

The known numerical criterion for stability can thus be stated:

\begin{thm}[Hilbert-Mumford criterion]
Let $G$ be a reductive group acting linearly on a projective variety $X\hookrightarrow \mathbb{P}^n$. Then for a point $x\in X$ we have that $x$ is semistable (resp. stable) if and only if $\mu(x,\lambda)\leq 0$ (resp. $<$) for all one-parameter subgroups $\lambda$ of $G$.
\label{HM}
\end{thm}

That is, a point $x\in X$ is unstable (resp. not stable) for the $G-$action if and only if there exists a one-parameter subgroup $\lambda:\mathbb{C}^{\times}\to G$ for which all the weights of $x$ are all positive (resp. non-negative).

In this paper we are interested in the case where $G$ is the group $SL(3)$ and $X$ is the space $\mathscr{P}_d$ of pencils of plane curves of degree $d$, embedded via Pl\"{u}cker coordinates in projective space.

\section{Stability Criterion for Pencils of Plane Curves}
\label{sc}

 Following the same approach as in \cite{stab}, we view a pencil of plane curves of degree $d$ as a choice of  line in the space of all plane curves of degree $d$. In other words, we identify the space $\mathscr{P}_d$ of all such pencils with the Grassmannian $Gr(2,S^dV^{\ast})$, where  $V \doteq H^0(\mathbb{P}^2,\mathcal{O}_{\mathbb{P}^2}(1))$. The latter, in turn, can be embedded in $\mathbb{P}(\Lambda^2 S^dV^{\ast})$ via Pl\"{u}cker coordinates. The group  $SL(V)$ acts naturally on $V$, hence on the invariant subvariety $\mathscr{P}_d$, and our goal is to describe the corresponding stability conditions. Since our main tool for that is criterion of Hilbert-Mumford, we need to know how the diagonal elements act on such coordinates. 

Concretely, choosing a pencil $\mathcal{P} \in \mathscr{P}_d$ and two curves $C_f$ and $C_g$ as generators, these represented (in some choice of coordinates) by $f=\sum f_{ij}x^iy^jz^{d-i-j}$ and $g=\sum g_{ij}x^iy^jz^{d-i-j}$ respectively, the Pl\"{u}cker coordinates of $\mathcal{P}$ are given by all the $2\times 2$ minors 
\[
m_{ijkl}\doteq \begin{vmatrix} f_{ij} & f_{kl} \\ g_{ij} & g_{kl} \end{vmatrix}
\]
Thus, the action of $\begin{pmatrix} \alpha & 0 & 0\\ 0 & \beta & 0\\ 0 & 0 & \gamma \end{pmatrix} \in SL(V)$ on the Pl\"{u}cker coordinates is given by 
\[
(m_{ijkl}) \mapsto (\alpha^{i+k}\beta^{j+l}\gamma^{2d-i-j-k-l}m_{ijkl})
\]

\subsection{The Stability of the Generators}
\label{stabgen}

It turns out that we are able to partially determine whether a pencil $\mathcal{P}\in \mathscr{P}_d$ is unstable (resp. not stable) or not by looking at the stability of its generators and, in particular, by looking at the log canonical threshold of its members. Therefore, from now on we will consider the actions of $SL(V)$ on both $\mathscr{P}_d$ and the space of plane curves of degree $d$.  

Our strategy consists in introducing an "affine" analogue of the Hilbert-Mumford weight (see Definition \ref{hmw}) and translate the numerical criterion of Hilbert-Mumford in terms of this quantity. More precisely, given a pencil $\mathcal{P}\in \mathscr{P}_d$ and a curve $C_f\in \mathcal{P}$, the idea is to use this affine weight to bound the log canonical threshold of the pair $(\mathbb{P}^2,C_f)$. The definition is as follows:

\begin{defi}
Given $\mathcal{P} \in \mathscr{P}_d$ and a one-parameter subgroup $\lambda: \mathbb{C}^{\times} \to SL(V)$ we define the affine weight of $\mathcal{P}$ at $\lambda$ to be
\[
\omega(\mathcal{P},\lambda)\doteq \min\{(a_x-a_z)(i+k)+(a_y-a_z)(j+l)\,\,:\,\,m_{ijkl}\neq 0\}
\]
\end{defi}

The inspiration for this definition comes from Definition 2.2 in \cite{kollarpol} and it is justified by Lemma \ref{boundlct}. The notations are the same as above and, even when omitted, we will always choose coordinates $[x,y,z]$ in $\mathbb{P}^2$ so that a one-parameter subgroup $\lambda$ is normalized. That is, it is given by
\begin{align}
\lambda: \mathbb{C}^{\times}  &\to SL(V) \nonumber\\
t &\mapsto\left( [x,y,z] \mapsto \begin{pmatrix} t^{a_x} & 0 & 0 \\ 0 & t^{a_y} & 0 \\ 0 & 0 & t^{a_z} \end{pmatrix}\cdot \begin {pmatrix}x\\y\\z \end{pmatrix}\right)
\label{normal}
\end{align}
for some weights $a_x,a_y,a_z \in \mathbb{Z}$ with $a_x \geq a_y \geq a_z, a_x>0$ and $a_x+a_y+a_z=0$. Then, stated in terms of $\omega(\mathcal{P},\lambda)$, the Hilbert-Mumford criterion becomes:

\begin{prop}
A pencil $\mathcal{P} \in \mathscr{P}_d$ is unstable (resp. not stable)  if and only if there exists a one-parameter subgroup $\lambda:\mathbb{C}^{\times} \to SL(V)$ and a choice of coordinates in $\mathbb{P}^2$ such that
\[
 \omega(\mathcal{P},\lambda)>\frac{2d}{3}(a_x+a_y-2a_z) \quad (\mbox{resp.}\,\,\geq )
\]
\label{unstable}
\end{prop}

\begin{proof}
A pencil $\mathcal{P} \in \mathscr{P}_d$ is unstable (resp. not stable)  if and only if there exists a one-parameter subgroup $\lambda: \mathbb{C}^{\times} \to SL(V)$ and a choice of coordinates in $\mathbb{P}^2$ satisfying that for any $i,j,k$ and $l$ such that $m_{ijkl}\neq 0$ (in those coordinates) we have
\[
a_x(i+k)+a_y(j+l)+a_z(2d-i-j-k-l)>0 \quad (\mbox{resp.}\,\, \geq 0)
\] 
if and only if
\begin{eqnarray*}
(a_x-a_z)(i+k)+(a_y-a_z)(j+l)-\frac{2d}{3}(a_x+a_y-2a_z)>0 \quad (\mbox{resp.}\,\, \geq 0)
\end{eqnarray*}
\end{proof}

Similarly, we define an affine weight for plane curves of degree $d$:

\begin{defi}
Given a plane curve of degree $d$ $C_f$  and a one-parameter subgroup \linebreak $\lambda: \mathbb{C}^{\times} \to SL(V)$ we define the affine weight of $f$ at $\lambda$ to be
\[
\omega(f,\lambda)\doteq \min\{(a_x-a_z)i+(a_y-a_z)j\,\,:\,\,f_{ij}\neq 0\}
\]
\end{defi}

And for curves the Hilbert-Mumford criterion becomes:

\begin{prop}
A curve $C_f$ is unstable (resp. not stable)  if and only if there exists a one-parameter subgroup $\lambda:\mathbb{C}^{\times} \to SL(V)$ and a choice of coordinates in $\mathbb{P}^2$ such that
\[
\omega(f,\lambda)>\frac{d}{3}(a_x+a_y-2a_z) \quad (\mbox{resp.}\,\,\geq )
\]
\label{unstablec}
\end{prop}

Given a pencil $\mathcal{P}\in \mathscr{P}_d$ and a curve $C_f\in \mathcal{P}$, it is interesting to compare the affine weights $\omega(f,\lambda)$ and $\omega(\mathcal{P},\lambda)$  for a fixed one-parameter subgroup $\lambda$. We state and prove a series of Propositions in this direction that allow us to relate the stability of a pencil to the stability of its generators.

\begin{prop}
Given a pencil $\mathcal{P} \in \mathscr{P}_d$ and any two (distinct) curves $C_f,C_g\in \mathcal{P}$ we have that
\[
\omega(f,\lambda) \leq \omega(f,\lambda) + \omega(g,\lambda)  \leq \omega(\mathcal{P},\lambda) 
\]
for all one-parameter subgroups $\lambda:\mathbb{C}^{\times} \to SL(V)$.
\label{ineq}
\end{prop}

\begin{proof}
Given $\mathcal{P}$ and  $\lambda:\mathbb{C}^{\times} \to SL(V)$, choose coordinates in $\mathbb{P}^2$ that normalize $\lambda$ and choose any two curves $C_f$ and $C_g$ of $\mathcal{P}$ so that $\mathcal{P}$ is represented by the Pl\"{u}cker coordinates $m_{ijkl}=f_{ij}g_{kl}-g_{ij}f_{kl}$.

Let $i,j,k$ and $l$ be such that $m_{ijkl}=f_{ij}g_{kl}-g_{ij}f_{kl}\neq 0$ and 
\[
\omega(\mathcal{P},\lambda)=(a_x-a_z)(i+k)+(a_y-a_z)(j+l)
\]
Then either $i$ and $j$ are such that $f_{ij}\neq0$ or $k$ and $l$ are such that $f_{kl}\neq 0$. In the first case there are two possibilities: either $g_{kl}=0$, which implies $g_{ij}\neq0$ and $f_{kl}\neq0$; or $g_{kl}\neq0$. Similarly, in the second case either $g_{ij}=0$, which implies $g_{kl}\neq0$ and $f_{ij}\neq0$; or $g_{ij}\neq0$.

In any case we have
\begin{eqnarray*}
(a_x-a_z)(i+k)+(a_y-a_z)(j+l)&=&\big((a_x-a_z)i+(a_y-a_z)j\big)+\\
& & + \big((a_x-a_z)k+(a_y-a_z)l\big)\\
&\geq& \omega(f,\lambda)+\omega(g,\lambda)
\end{eqnarray*}
\end{proof}

\begin{prop}
Given $\mathcal{P} \in \mathscr{P}_d$, a one-parameter subgroup $\lambda:\mathbb{C}^{\times} \to SL(V)$ and any curve $C_f\in \mathcal{P}$ there exists a curve $C_g$ in $\mathcal{P}$ such that
\[
\omega(\mathcal{P},\lambda) \leq \omega(f,\lambda) + \omega(g,\lambda) 
\]
\label{twocurves}
\end{prop}

\begin{proof}
Fix  $\lambda:\mathbb{C}^{\times} \to SL(V)$ and coordinates in $\mathbb{P}^2$ that normalize $\lambda$. Choose any two curves $C_f$ and $C_g$ of $\mathcal{P}$. Let $i$ and $j$ be such that $f_{ij}\neq 0$ and 
\[
\omega(f,\lambda)=(a_x-a_z)i+(a_y-a_z)j
\]
Replacing $g$ by $g'=g-\frac{g_{ij}}{f_{ij}}f$ we have $g_{ij}=0$, hence $m_{ijkl}\neq 0$ for all $k$ and $l$ such that $g_{kl}\neq 0$ and it follows that
\[
\omega(\mathcal{P},\lambda) \leq \omega(f,\lambda)+\omega(g,\lambda)
\]
\end{proof}

\begin{cor}
Given $\mathcal{P} \in \mathscr{P}_d$, a one-parameter subgroup $\lambda:\mathbb{C}^{\times} \to SL(V)$ and any curve $C_f\in \mathcal{P}$ there exists a curve $C_g$ in $\mathcal{P}$ such that
\[
\omega(\mathcal{P},\lambda) \leq 2 \max\{\omega(f,\lambda),\omega(g,\lambda) \}
\]
\label{ctwocurves}
\end{cor}

\begin{cor}
Given $\mathcal{P} \in \mathscr{P}_d$, a one-parameter subgroup $\lambda:\mathbb{C}^{\times} \to SL(V)$ and any curve $C_f\in \mathcal{P}$ there exists a curve $C_g$ in $\mathcal{P}$ such that
\[
\omega(\mathcal{P},\lambda) = \omega(f,\lambda) + \omega(g,\lambda)
\]
\end{cor}

\begin{cor}
If a pencil $\mathcal{P}\in \mathscr{P}_d$ has only semistable (resp. stable) members, then $\mathcal{P}$ is semistable (resp. stable).
\label{allmembers}
\end{cor}

\begin{cor}
If a pencil $\mathcal{P}\in \mathscr{P}_d$ contains only plane curves $C_d$ such that the pairs $\left(\mathbb{P}^2,3/dC_d\right)$ (resp.$\left(\mathbb{P}^2,\left(3/d+\varepsilon\right) C_d\right), 0<\varepsilon << 1$) are log canonical, then  $\mathcal{P}$ is semistable (resp. stable).
\label{main}
\end{cor}

\begin{proof}
As observed in \cite{hacking} and \cite{kimlee}, in this case all members of $\mathcal{P}$ are semistable (resp. stable).
\end{proof}

As a result of our comparison between $\omega(f,\lambda)$ and $\omega(\mathcal{P},\lambda)$ we prove Theorems \ref{1sss} and \ref{2sss} below:

\begin{thm}
If $\mathcal{P}\in \mathscr{P}_d$ contains at worst one strictly semistable curve (and all other curves in $\mathcal{P}$ are stable), then $\mathcal{P}$ is stable.
\label{1sss}
\end{thm}

\begin{proof}
Given $\mathcal{P}$ as above, if all curves in $\mathcal{P}$ are stable, then $\mathcal{P}$ is stable by Corollary \ref{allmembers}. Otherwise, let $C_f$ be the unique strictly semistable curve in $\mathcal{P}$. Given any one-parameter subgroup $\lambda$, by Proposition \ref{twocurves} there exists a curve $C_g$ such that
\[
\frac{\omega(\mathcal{P},\lambda)}{(a_x-a_z)+(a_y-a_z)}\leq \frac{\omega(f,\lambda)}{(a_x-a_z)+(a_y-a_z)}+ \frac{\omega(g,\lambda)}{(a_x-a_z)+(a_y-a_z)}
\]
And because $C_f$ (resp. $C_g$) is strictly semistable (resp. stable) it follows that
\[
 \frac{\omega(f,\lambda)}{(a_x-a_z)+(a_y-a_z)}\leq \frac{d}{3} \qquad \mbox{and} \qquad \frac{\omega(h,\lambda)}{(a_x-a_z)+(a_y-a_z)}<\frac{d}{3}
\]
and hence
\[
\frac{\omega(\mathcal{P},\lambda)}{(a_x-a_z)+(a_y-a_z)}<\frac{2d}{3}
\]
That is, $\mathcal{P}$ is stable.
\end{proof}

\begin{thm}
If $\mathcal{P}\in \mathscr{P}_d$ contains at worst two semistable curves $C_f$ and $C_g$ (and all other curves in $\mathcal{P}$ are stable), then $\mathcal{P}$ is strictly semistable if and only if there exists a one-parameter subgroup $\lambda$ (and coordinates in $\mathbb{P}^2$) such that $C_f$ and $C_g$ are both non-stable with respect to this $\lambda$ that is,
\[
 \frac{\omega(f,\lambda)}{(a_x-a_z)+(a_y-a_z)}= \frac{d}{3} \qquad \mbox{and} \qquad \frac{\omega(g,\lambda)}{(a_x-a_z)+(a_y-a_z)}=\frac{d}{3}
\]
\label{2sss}
\end{thm}

\begin{proof}
Fix $\mathcal{P}$ as above and note that $\mathcal{P}$ is semistable (Corollary \ref{allmembers}). First, note that if the two inequalities above hold for some $\lambda$, then $\mathcal{P}$ is strictly semistable by Proposition \ref{ineq}. Thus, assume $\mathcal{P}$ is strictly semistable. Then there exists a one-parameter subgroup $\lambda$  (and coordinates in $\mathbb{P}^2$) such that 
\[
 \frac{\omega(\mathcal{P},\lambda)}{(a_x-a_z)+(a_y-a_z)}= \frac{2d}{3}
\]
and, by Corollary \ref{ctwocurves}, it must exist a curve $C_h$ in $\mathcal{P}$ such that
\[
\frac{d}{3} \leq   \max\left\{\frac{\omega(f,\lambda)}{(a_x-a_z)+(a_y-a_z)}, \frac{\omega(h,\lambda)}{(a_x-a_z)+(a_y-a_z)}\right\} 
\]
In particular, either $C_f$ or $C_h$ is non-stable with respect to this $\lambda$. But $C_f$ and $C_g$ are the only potentially non-stable curves in $\mathcal{P}$. Therefore, either
\begin{equation}
 \frac{\omega(f,\lambda)}{(a_x-a_z)+(a_y-a_z)}\geq \frac{d}{3}
 \label{i1}
\end{equation}
or $C_h=C_g$ and 
\begin{equation}
 \frac{\omega(g,\lambda)}{(a_x-a_z)+(a_y-a_z)}\geq \frac{d}{3}
 \label{i2}
\end{equation}
In any case, we claim that the following two equalities hold
\[
 \frac{\omega(f,\lambda)}{(a_x-a_z)+(a_y-a_z)}= \frac{d}{3} \qquad \mbox{and} \qquad \frac{\omega(g,\lambda)}{(a_x-a_z)+(a_y-a_z)}=\frac{d}{3}
\]
In fact, if $C_h=C_g$ and (\ref{i2}) holds, then  
\[
 \frac{\omega(g,\lambda)}{(a_x-a_z)+(a_y-a_z)}= \frac{d}{3}
\]
because $C_g$ is semistable. Thus, by Proposition \ref{twocurves}, inequality (\ref{i1}) must be true also. 

Now, if (\ref{i1}) holds, then
\[
 \frac{\omega(f,\lambda)}{(a_x-a_z)+(a_y-a_z)}= \frac{d}{3}
\]
because $C_f$ is semistable. Thus, by Proposition \ref{twocurves}, we have that 
\[
\frac{\omega(h,\lambda)}{(a_x-a_z)+(a_y-a_z)}\geq \frac{d}{3}
\]
and, by assumption, it must be the case that $C_h=C_g$ (and (\ref{i2}) holds).

\end{proof}

\section{Stability and the log canonical threshold}
\label{stablct}

We are now ready to describe how $\omega(\mathcal{P},\lambda)$ and $\omega(f,\lambda)$ are related to the log canonical threshold of the pair $(\mathbb{P}^2,C_f)$. We begin by proving the following:

\begin{prop}
Given $\mathcal{P}\in \mathscr{P}_d$ and any base point $p$ of $\mathcal{P}$, there exists a one-parameter subgroup $\lambda: \mathbb{C}^{\times} \to SL(V)$ (and coordinates in $\mathbb{P}^2$) such that for any curve $C_f$ in $\mathcal{P}$ we have that 
\[
\frac{(a_x-a_z)+(a_y-a_z)}{\omega(\mathcal{P},\lambda)}\leq lct_p(\mathbb{P}^2,C_f)
\]
\label{boundlct2}
\end{prop}

\begin{proof}
Given $\mathcal{P}$ and a base point $p$, we can always choose coordinates in $\mathbb{P}^2$ so that $p= (0:0:1)$.

Given any $a\in \mathbb{Q}\cap (-1/2,1]$, we can let $a_x=1,a_y=a$ and $a_z=-1-a$ and consider the one-parameter subgroup $\lambda$, which in these coordinates is normalized, as in (\ref{normal}). Then
\[
\frac{(a_x-a_z)+(a_y-a_z)}{\omega(\mathcal{P},\lambda)}=\frac{3(1+a)}{(2+a)(i+k)+(2a+1)(j+l)}
\]
for some $0\leq i,j,k,l \leq d$ such that $m_{ijkl}\neq 0$. 

Because $f_{00}=0$ for any curve $C_f$ in $\mathcal{P}$, we have that $m_{00kl}=0$ for all $0\leq k,l \leq d$. This implies
\[
\frac{3(1+a)}{(2+a)(i+k)+(2a+1)(j+l)} \leq 1
\]
for all $i,j,k,l$ such that $m_{ijkl}\neq 0$.

We claim that given $a\in \mathbb{Q}\cap (-1/2,1]$, the corresponding one-parameter subgroup $\lambda$ is such that for any curve $C_f$ in $\mathcal{P}$ we have
\[
\frac{(a_x-a_z)+(a_y-a_z)}{\omega(P,\lambda)}\leq lct_p(\mathbb{P}^2,C_f)
\]

By contradiction, assume there exists $C_f$ in $\mathcal{P}$ such that
\[
lct_p(\mathbb{P}^2,C_f) < \frac{(a_x-a_z)+(a_y-a_z)}{\omega(\mathcal{P},\lambda)}
\]

Write $\tilde{f}(u,v)=f(x,y,1)$ and assign weights $\omega(u)\doteq a_x-a_z=2+a$ to the variable $u$ and $\omega(v)\doteq a_y-a_z=2a+1$ to the variable $v$ so that the weighted multiplicity of $\tilde{f}$ is precisely $\omega(f,\lambda)$.  

Now, consider the finite morphism $\varphi:\mathbb{C}^2\to \mathbb{C}^2$ given by $(u,v)\mapsto (u^{\omega(u)},v^{\omega(v)})$ and let 
\[
\Delta\doteq (1-\omega(u))H_u+(1-\omega(v))H_v+ c\cdot \tilde{f}(u^{\omega(u)},v^{\omega(v)})
\]
 where $H_u$ (resp. $H_v$) is the divisor of $u=0$ (resp. $v=0$) and $c \in \mathbb{Q}\cap [0,1]$. Then
\[
\varphi^*(K_{\mathbb{C}^2}+c \cdot \tilde{f}(u,v))=K_{\mathbb{C}^2}+\Delta
\]
and by Proposition 5.20 (4) in \cite{km} we know that the pair $(\mathbb{C}^2,c \cdot \tilde{f})$ is log canonical at $(0,0)$ if and only if the pair $(\mathbb{C}^2,\Delta)$ is log canonical at $(0,0)$.

In particular, taking $ c=\frac{\omega(u)+\omega(v)}{\omega(\mathcal{P},\lambda)}> lct_p(\mathbb{P}^2,C_f)=lct_0(\mathbb{C}^2,\tilde{f})$ it follows that
\[
a(E;\mathbb{C}^2,\Delta)=-1+\omega(u)+\omega(v)-c \cdot \omega(f,\lambda)<-1
\]
where $E$ is the exceptional divisor of the blow-up of $\mathbb{C}^2$ at the origin and $a(E;\mathbb{C}^2,\Delta)$ is the corresponding discrepancy. But the above inequality is equivalent to the inequality $\omega(\mathcal{P},\lambda)<\omega(f,\lambda)$, which contradicts Proposition \ref{ineq}.
\end{proof}

Next, we recall the following known result:

\begin{lema}[\cite{singpairs}*{Proposition 8.13}]
Let $C_f$ be any plane curve. Then
\begin{equation}
\frac{\omega(f,\lambda)}{(a_x-a_z)+(a_y-a_z)}\leq \frac{1}{lct(\mathbb{P}^2,C_f)}
\label{ineqlct1}
\end{equation}
for any one-parameter subgroup $\lambda: \mathbb{C}^{\times} \to SL(V)$. 
\label{boundlct}
\end{lema}

\begin{proof}
Fix any one-parameter subgroup $\lambda: \mathbb{C}^{\times} \to SL(V)$ and choose coordinates in $\mathbb{P}^2$ so that $\lambda$ is normalized, as in (\ref{normal}). There are two possibilities: either $a_y>a_z$ (hence $a_x>a_z$) or $a_y=a_z$. Let us first consider the former.

If $p\doteq (0,0,1)\notin C_f$, then $f_{00}\neq0$, which implies $\omega(f,\lambda)=0$ and inequality (\ref{ineqlct1}) is true. Otherwise, we can write $\tilde{f}(u,v)=f(x,y,1)$ and assign weights $\omega(x)=a_x$ to the variable $x$, $\omega(y)=a_y$ to the variable $y$ and $\omega(z)=a_z$ to the variable $z$. Then $u$ has weight $a_x-a_z$, $v$ has weight $a_y-a_z$ and we have that the weighted multiplicity of $\tilde{f}$ is precisely $\omega(f,\lambda)$.

Proposition 8.13 in \cite{singpairs} tells us
\[
\frac{\omega(f,\lambda)}{(a_x-a_z)+(a_y-a_z)}\leq \frac{1}{lct_0(\mathbb{C}^2,\tilde{f})}
\]
and the result follows from the fact that $lct(\mathbb{P}^2,C_f) \leq lct_p(\mathbb{P}^2,C_f)=lct_0(\mathbb{C}^2,\tilde{f})$.

Finally, if we are in the situation when $a_y=a_z$, then 
\[
\omega(f,\lambda)=\min \{(a_x-a_z)i\,\,;\,\, f_{ij}\neq 0\}
\]
 and the desired inequality becomes
\[
 c\doteq \min \{i\,\,;\,\,f_{ij}\neq 0\}\leq \frac{1}{lct(\mathbb{P}^2,C_f)}
\]

If $c=0$ or $c=1$ the inequality is obvious. And if $c\geq 2$, then $C_f$ contains a line $(x=0)$ with multiplicity $c\geq 2$ and, again, the inequality is true. 
\end{proof}

In particular, we conclude from Corollary \ref{ctwocurves} that:

\begin{prop}
Given a pencil $\mathcal{P}\in \mathscr{P}_d$ we have that for any one-parameter subgroup $\lambda: \mathbb{C}^{\times}\to SL(V)$ there exists $C_f\in \mathcal{P}$ such that
\begin{equation}
 \frac{\omega(\mathcal{P},\lambda)} {(a_x-a_z)+(a_y-a_z)} \leq \frac{2}{lct(\mathbb{P}^2,C_f)}
\label{ineqlct}
\end{equation}
\label{rellctomega}
\end{prop}

And, as a consequence, we recover the statement from Corollary \ref{main}:

\begin{cor}
If $\mathcal{P}\in \mathscr{P}_d$ is a pencil such that $lct(\mathbb{P}^2,C_f)\geq 3/d$ (resp. $>3/d$) for any curve $C_f$ in $\mathcal{P}$, then $\mathcal{P}$ is semistable (resp. stable).
\end{cor}

Proposition \ref{boundlct2} and Lemma \ref{boundlct} together with the other results obtained in this section, allow us to prove Theorems \ref{alpha} and \ref{onequarter} below. Both results relate the stability of $\mathcal{P}$ and the log canonical threshold of the pair $(\mathbb{P}^2,C_f)$ for $C_f\in \mathcal{P}$.

\begin{thm}
Let $\mathcal{P}$ be a pencil in $\mathscr{P}_d$ which contains a curve $C_f$ such that $lct(\mathbb{P}^2,C_f)=\alpha$. If $\mathcal{P}$ is unstable (resp. not stable), then $\mathcal{P}$ contains a curve $C_g$ such that $
lct(\mathbb{P}^2,C_g)< \frac{3\alpha}{2d\alpha-3}$ (resp. $\leq$).
\label{alpha}
\end{thm}

\begin{proof}
 If $\mathcal{P}$ is unstable (resp. not stable), then by Proposition \ref{unstable} we can choose a one-parameter subgroup $\lambda$ (and coordinates in $\mathbb{P}^2$) so that
\[
\frac{2d}{3} < \frac{\omega(\mathcal{P},\lambda)}{(a_x-a_z)+(a_y-a_z)} \qquad (\mbox{resp.}\,\,\leq)
\]
By Proposition \ref{twocurves}, we can find a  a curve $C_g$ in $\mathcal{P}$ such that
\[
\frac{\omega(\mathcal{P},\lambda)}{(a_x-a_z)+(a_y-a_z)} \leq \frac{\omega(f,\lambda)}{(a_x-a_z)+(a_y-a_z)}+ \frac{\omega(g,\lambda)}{(a_x-a_z)+(a_y-a_z)}
\]
Moreover, by Lemma \ref{boundlct} we have that
\[
\frac{\omega(f,\lambda)}{(a_x-a_z)+(a_y-a_z)}\leq \frac{1}{lct(\mathbb{P}^2,C_f)} \quad \mbox{and} \quad \frac{\omega(g,\lambda)}{(a_x-a_z)+(a_y-a_z)}\leq \frac{1}{lct(\mathbb{P}^2,C_g)}
\]
And, because $lct(\mathbb{P}^2,C_f)=\alpha$, combining the above inequalities we conclude that
\[
\frac{2d}{3}-\frac{1}{\alpha} < \frac{1}{lct(\mathbb{P}^2,C_g)} \qquad (\mbox{resp.}\,\,\leq) \iff lct(\mathbb{P}^2,C_g)< \frac{3\alpha}{2d\alpha-3} \qquad (\mbox{resp.}\,\, \leq )
\]
\end{proof}

\begin{thm}
If $\mathcal{P}\in \mathscr{P}_d$ is semistable (resp. stable), then for any curve $C_f$ in $\mathcal{P}$ and any base point $p$ of $\mathcal{P}$ we have $\frac{3}{2d} \leq lct_p(\mathbb{P}^2,C_f)$ (resp. $<$).
\label{onequarter}
\end{thm}

\begin{proof}
Fix $\mathcal{P}\in \mathscr{P}_d$ and a base point $p$ as above. Given $C_f$ we can always find coordinates in $\mathbb{P}^2$ so that $p=(0:0:1)$ and we can choose $\lambda$ as in Proposition \ref{boundlct2}. Because $\mathcal{P}$ is semistable (resp. stable) for this $\lambda$ we have that
\[
\frac{3}{2d}\leq \frac{(a_x-a_z)+(a_y-a_z)}{\omega(\mathcal{P},\lambda)} \qquad (\mbox{resp.}\,\,<)
\]
and the result follows from Proposition \ref{boundlct2}.
\end{proof}

\begin{rmk}
It is important to observe that Theorems \ref{alpha} and \ref{onequarter} can be easily generalized for hypersurfaces of degree $d$ in $\mathbb{P}^n$. The corresponding statements are presented below.
\label{hypersurfaces}
\end{rmk}

\begin{thm}[Analogue of Theorem \ref{alpha}]
Let $\mathcal{P}$ be a pencil of hypersurfaces of degree $d$ in $\mathbb{P}^n$ which contains a hypersurface $F(f=0)$ such that $lct(\mathbb{P}^n,F)=\alpha$. If $\mathcal{P}$ is unstable (resp. not stable), then $\mathcal{P}$ contains a hypersurface $G(g=0)$ such that $lct(\mathbb{P}^n,G)< \frac{\alpha (n+1)}{2d\alpha-(n+1)}$ (resp. $\leq \frac{\alpha(n+1)}{2d\alpha-(n+1)}$).
\end{thm}

\begin{thm}[Analogue of Theorem \ref{onequarter}]
If $\mathcal{P}$ is semi-stable (resp. stable), then for any hypersurface $F(f=0)$ in $\mathcal{P}$ and any base point $p$ of $\mathcal{P}$ we have that
\[
\frac{n+1}{2d} \leq lct_p(\mathbb{P}^{n},F) \qquad (\mbox{resp.}\,\, <)
\]
\end{thm}

\section{Stability and the multiplicity at a base point}
\label{sm}

We now relate $\omega(\mathcal{P},\lambda)$ to the multiplicity of the generators of $\mathcal{P}$ at a base point. Our result is the following: 

\begin{thm}
Let $\mathcal{P}$ be a pencil in $\mathscr{P}_d$ with generators $C_f$ and $C_g$. If there exists a base point $P$ of $\mathcal{P}$ such that $\text{mult}_P(C_f)+\text{mult}_P(C_g)>\frac{4d}{3}$ (resp. $\geq$), then $\mathcal{P}$ is unstable (resp. not stable).
\label{unsmult}
\end{thm}

\begin{proof}
If $P$ is any base point of $\mathcal{P}$, we can always choose coordinates so that we have $P=(0:0:1)$. Let $a_x=1,a_y=1,a_z=-2$ and $\lambda$ be the one-parameter subgroup which in these coordinates is normalized as in (\ref{normal}). Then $\omega(f,\lambda)=3\cdot \text{mult}_P(C_f)$ and $\omega(g,\lambda)=3\cdot \text{mult}_P(C_g)$ for any choice of generators of $\mathcal{P}$, say $C_f$ and $C_g$. These two equalities, together with Proposition \ref{ineq}, imply
\[
\frac{\omega(\mathcal{P},\lambda)}{(a_x-a_z)+(a_y-a_z)}\geq 3\cdot \frac{\text{mult}_P(C_f)+\text{mult}_P(C_g)}{(a_x-a_z)+(a_y-a_z)}
\]
And since $(a_x-a_z)+(a_y-a_z)=6$, the result then follows from the Hilbert-Mumford criterion (Proposition \ref{unstable}).
\end{proof}

\bibliography{references}

\end{document}